\documentclass[11 pt,a4paper]{amsart} 
\usepackage{amsfonts,amssymb,amscd,amsmath,enumerate,verbatim,calc} 
\usepackage{float}
\usepackage{amsthm}
\newtheorem{theorem}{Theorem}[section]
\newtheorem{definition}{Definition}[section]

\newtheorem{corollary}[theorem]{Corollary}
\newtheorem{proposition}[theorem]{Proposition}
\newtheorem{remark}[theorem]{Remark}
\usepackage{mathtools}

\numberwithin{equation}{section}

\usepackage{amsfonts}
\newcommand{\field}[1]{\mathbb{#1}}

                   \newcommand{\Z}{\field{Z}}
\newcommand{\C}{\field{C}}

\newcommand{\im}{\rm Im}

\renewcommand{\ker}{\textnormal{ker}}
\renewcommand{\im}{\textnormal{Im}}

\begin{document}
	
\title{The Schur multiplier  of some finite multiplicative Lie algebras}
	
\author{Amit Kumar$^{1}$, Deepak Pal$^{2}$, Seema Kushwaha$^{3}$ AND Sumit Kumar Upadhyay$^{4}$\vspace{.4cm}\\
{Department of Applied Sciences,\\ Indian Institute of Information Technology Allahabad\\Prayagraj, U. P., India} }

\thanks{$^1$amitkroor98@gmail.com, $^2$deepakpal5797@gmail.com, $^3$seema28k@gmail.com, $^4$upadhyaysumit365@gmail.com}
	
%
	

\begin{abstract}
The main aim of the article is to find the Schur multiplier and the Lie exterior square of some finite multiplicative Lie algebras. For a non abelian simple group $K$ with  trivial Schur multiplier,  we  see that the Schur multiplier of  multiplicative Lie algebra $K$ is trivial  and the Lie exterior square of $K$ is an improper  multiplicative Lie algebra $K.$ 
\end{abstract}
\maketitle
\section{Introduction}The Schur multiplier of a group was introduced by Schur in 1904 during his works on projective representations of groups \cite{Schur}. Given a group $K,$ the Schur multiplier of $K$ is the second cohomology group $ H^2(K,\C^*),$ where $K$ acts trivially on $\C^*.$    Several attempts have been done to characterize the structure of the Schur multiplier for some classes of groups, for example, the dihedral group, metacyclic group, alternating group, quaternion group, and many other groups \cite{K87}.  It has been a powerful tool in the classification of finite simple groups, and finite $p$-groups \cite{GJ99,  K87,NM12}. It is worth to describe the structure of the Schur multiplier for a given group.

The finite dimensional Lie algebra analogue to the Schur multiplier was developed in \cite{PB93,PB96}. Again, the classification of finite dimensional nilpotent Lie algebras according to the dimension of its Schur multiplier has been a line of investigation. 

In 1993, G. J. Ellis \cite{GJ}  introduced an interesting algebraic structure called “multiplicative Lie algebra” which is a generalization of groups as well as Lie algebras. 
In \cite{FP}, F. Point
and P. Wantiez introduced the concept of nilpotency for multiplicative Lie algebras and
proved many nilpotency results which are known for groups and Lie algebras. 
In 2021, M. S. Pandey et al. \cite{ML} discussed another notion of solvable and nilpotent multiplicative Lie algebra with the help of a multiplicative Lie center and Lie commutator.
To
see more structural properties of multiplicative Lie algebra, in \cite{AGNM}  Bak et al.
constructed and studied two homology theories for multiplicative Lie algebras which
are already known for groups and Lie algebras. The authors introduced and studied
Schur multiplier in terms of homology. In 2018, Lal and Upadhyay studied the theory of extensions of multiplicative Lie algebras and introduced the Schur multiplier as the second cohomology of multiplicative Lie algebras \cite{RLS}. The authors also introduced the concept of the Lie exterior square of multiplicative Lie algebras and consequently, they described the Schur multiplier of a multiplicative Lie algebra as the group of non-trivial relations satisfied by the Lie product. Schur multiplier was also expressed in terms of presentations. 

It is an interesting problem to characterize the structure of the Schur multiplier for multiplicative Lie algebras. Note that, given a group $K$, there is always a trivial multiplicative Lie algebra structure on $K$. Every non abelian group holds at least two multiplicative Lie algebra structures: one is trivial and the other one is given by a commutator. In \cite{MS}, Pandey et al. determined all possible multiplicative Lie algebra structures for certain finite groups.
In this paper, we find the Schur multipliers and the Lie exterior squares of certain classes of finite multiplicative Lie algebras given in \cite{MS}. At the end of the article, we give a table having the Schur multiplier of some precise multiplicative Lie algebras.

  \section{Preliminaries}
In this section we recall some basics of the Schur multiplier of a group from \cite{K87},  multiplicative Lie algebras and second Lie cohomology from \cite{RLS}. 

\begin{definition}\label{Cohomology}
	Let $K$ and $H$ be a multiplicative group and an abelian multiplicative group, respectively. We say that $K$ acts on $H$ if for each $x \in  K$ and $a  \in  H$, there is a unique element of $H$ denoted by $ a^x $ or $ax$, such that 
	\begin{enumerate}
		\item $ (ab)^x=a^xb^x $  $( a, b  \in  H,$ $x  \in  K )$  
		\item $ (a^x)^y=a^{xy} $ $( a \in  H,$ $ x, y  \in  K ) $ 
		\item $ a^1=a $ $( a \in H )$ .
	\end{enumerate}
A function $f:K^n\to H$  is called an $n$-cochain of $K$ in $H$. The set of all $n$-cochains, written as $ C^n(K,H), $ becomes an abelian group under the multiplication. For $ n=0, $ we put  $ C^0(K,H)=H. $ The formula 
\begin{align*} 
	d_{n+1}(f(x_1,\dots,x_{n+1}))&=f(x_2,\dots,x_{n+1})\times  \prod_{i=1}^n f(x_1,\dots,x_{i-1},x_ix_{i+1},x_{i+2},\dots,x_{n+1})^{(-1)^i}\\
& ~~~~~~~~~~~~~~~~\times f(x_1,\dots,x_n)^{(-1)^{n+1}}x_{n+1} 
\end{align*}
determines a group homomorphism $$ d_{n+1}:C^n(K,H) \longrightarrow  C^{n+1}(K,H). $$
We have a sequence of homomorphisms of abelian groups:
$$ \cdots C^{n-1}(K,H) \xrightarrow[\text{}]{\text{d}_n} C^n(K,H) \xrightarrow[\text{}]{\text{d}_{n+1}}  C^{n+1}(K,H) \cdots $$
such that $ d_{n+1}\circ d_n$ is the zero map from  $C^{n-1}(K,H)$ to $C^{n+1}(K,H).$
So, $ \im (d_n) \subseteq \ker (d_{n+1}). $
Denote $ Z^n(K,H)=\ker (d_{n+1}) $ and $ B^n(K,H)=\im (d_n). $  The elements of $ Z^n(K,H) $ and $ B^n(K,H)$ are called $n$-cocycles and $n$-coboundaries,respectively.
The nth cohomology group $ H^n(K,H) $ of $K$ with coefficients in $H$ is defined by: $$ H^n(K,H) = \frac{Z^n(K,H)}{B^n(K,H)} , (n\geq1).  $$	
\end{definition}

\begin{definition}
The second cohomology group  $ H^2(K,\C^*) $ is called the  Schur multiplier of the group $K$  and is denoted by $ M(K)$, where $K$ denotes a finite group and  the action of $K$ on $\C^*$ is trivial.	
\end{definition}

Now we state the following results:

\begin{proposition}
	$ M(K) = \{1\}$, if $K$ is finite cyclic group.
\end{proposition}

\begin{proposition}\label{direct}
	Let $K_1$ and $K_2$ be finite abelian groups. Then $$ M(K_1\times K_2)\cong M(K_1)\times M(K_2)\times	Hom (K_1,K_2). $$
\end{proposition}

\begin{definition}
	Let $(K,\cdot)$ be a group. Then the non-abelian exterior square $ K\wedge K $ of $K$ is a group generated by the elements of the set $ \{a \wedge b:a,b\in K\} $ satisfying the following conditions:
	\begin{enumerate}
		\item $ a \wedge a=1 $
		\item $ ab \wedge c=(^ab \wedge ^ac)  (a \wedge c) $
		\item $ a \wedge bc=(a \wedge b)  (^ba \wedge ^bc) $
	\end{enumerate}
for all $a, b, c \in K. $
\end{definition}

Now, we have an epimorphism $ \chi : K\wedge K \longrightarrow [K,K] $ defined by
$ \chi(g\wedge h) = [g,h] $ and the kernel of $ \chi $ is the Schur multiplier $M(K)$ of $K$, that is, we have the short exact sequence 
\begin{equation}\label{1}
 {1}\longrightarrow  M(K)\longrightarrow K\wedge K \overset{\chi} \longrightarrow [K,K] \longrightarrow {1}. 
\end{equation}

\begin{definition}
A multiplicative Lie algebra is a triple $ (K,\cdot,\star), $ where $K$ is a set, $ \cdot $ and $ \star $ are two binary operations on $K$ such that $ (K,\cdot) $ is a group (need not be abelian) and for all  $x, y, z  \in  K$, the following identities hold: 
\begin{enumerate}
	\item $ x\star x=1 $  
	\item $ x\star(y \cdot z)=(x\star y)\cdot^y(x\star z) $ 
	\item $ (x \cdot y)\star z= ^x(y\star z) \cdot (x\star z) $ 
	\item $ ((x\star y)\star ^yz)((y\star z)\star^zx)((z\star x)\star^xy)=1 $ 
	\item $ ^z(x\star y)=(^zx\star ^zy)$ 
\end{enumerate}
where $^xy$ denotes $xyx^{-1}.$ We say that $ \star $ is a multiplicative Lie algebra structure on the group $K$.
\end{definition}

\begin{proposition}\cite{GJ} \label{MLA}
Let $(K,\cdot,\star)$ be a  multiplicative Lie algebra. Then the following further identities hold:
\begin{enumerate}
	\item $ (1\star x) = 1 = (x\star 1) $ for all $x\in K.$
	
	\item $ (x\star y)(y\star x) = 1 $  for all $x,y \in K.$
	
	\item $ ^{(x\star y)}(u\star v) = ^{[x,y]}(u\star v) $ for all $x,y,u,v \in K.$
	
	\item $ [(x\star y),z] = ([x,y]\star z) $ for all  $x,y,z \in K.$
	
	\item $ x^{-1}\star y = ^{x^{-1}}((x\star y)^{-1})$ and $ x\star y^{-1} = ^{y^{-1}}((x\star y)^{-1}) $ for all $x,y \in K.$
	
\end{enumerate}
\end{proposition}

By the universal  property of $ K\wedge K, $ there is a unique group homomorphism  
$$ \phi :K\wedge K \longrightarrow K $$ given by 
$ \phi (a \wedge b)=a \ast b $ 
whose image is $ K\ast K $ and  we have the short exact sequence

\begin{equation}\label{2}
{1}\longrightarrow  \ker(\phi)\longrightarrow K\wedge K \overset{\phi} \longrightarrow K\star K \longrightarrow {1}.
\end{equation}

\begin{definition}
	A multiplicative Lie algebra structure $\star$ on a group $K$ is said to be a proper  multiplicative Lie algebra structure if $a\star b \neq 1$ for some pair of elements $a,b \in K$
	and also $a\star b \neq [a,b] $ for some pair of elements $a,b \in K.$ We term a group $K$ as a Lie simple group if there is no  proper  multiplicative Lie algebra structure on $K.$

We say the multiplicative Lie algebra structure $\star$ is trivial if $a\star b = 1$ for every $a,b \in K$ and is trivial if $a\star b = [a, b]$ for every $a,b \in K$
\end{definition}

\begin{proposition}[Corollary 2.7 \cite{RLS}]\label{equivarient}
	Let $K$ be a group with trivial Schur multiplier. Then any multiplicative Lie algebra structure on $K$ is uniquely determined by a $K$-equivariant homomorphism from $[K,K]$ to $K.$ 
\end{proposition}

\begin{definition}\label{cocycle}
A multiplicative Lie 2-cocycle of a multiplicative Lie algebra $K$ with coefficients in an abelian group $H$ with trivial Lie product is a pair $(f,h)$, where $ f \in Z^2(K,H) $ is a group 2-cocycle of $K$ with coefficient in the trivial $K$-module $H$, and $h$ is a map from $ K \times K $ to $H$ satisfying the equations:
 
\begin{enumerate}
	\item $ h(x,1)=h(1,x)=h(x,x)=1 $ for all $ x \in K. $
	
	\item $ h(x,yz)=h(x,y)h(x,z)f(y^{-1},y)^{-1}f(y,x\star z) f(y(x\star z),y^{-1})f(x\star y,^y(x\star z)) $ for all $x, y, z \in K.$
	
	\item $ h(xy,z)=h(y,z)h(x,z)f(x^{-1},x)^{-1}f(x,y\star z) f(x(y\star z),x^{-1})f(^x(y\star z),(x\star z)) $ for all $x, y, z \in K.$
	
	\item $ h(y\star x,^xz)h(x\star z,^zy)h(z\star y,^yx) f((y\star x)\star ^xz,(x\star z)\star ^zy)f((y\star x)\star ^xz)\cdot(x\star z)\star ^zy, \linebreak (z\star y)\star ^yx)=1. $ for all $x, y, z \in K.$
	
	\item $ h(^zx,^zy)=h(x,y)f(z,x\star y)f(z^{-1},z)^{-1}f(z(x\star y),z^{-1}). $ for all $x, y, z \in K.$
\end{enumerate}

The set $ Z^2_{ML}(K,H) $ of all multiplicative Lie 2-cocycles of $K$ with coefficient in $H$ is an abelian group with respect to the coordinate wise operation given by $$ (f,h)\cdot (f',h')=(ff',hh'). $$
Given any identity preserving map $g$ from $K$ to $H$, the pair $(\delta g,g^*) $ of maps from $ K\times K $ to H given by $$\delta g(x,y)=g(y)g(xy)^{-1}g(x)$$ and $$g^*(x,y)=g(x\star y)^{-1}$$ is a member of $ Z^2_{ML}(K,H). $
Let $MAP(K,H)$ denote the group of identity preserving maps from $K$ to $H$. We have a group homomorphism $$\chi: MAP(K,H) \longrightarrow Z^2_{ML}(K,H) $$ given by  $$\chi(g)=(\delta g,g^*). $$
The image of $\chi$ is called the group of multiplicative Lie 2-coboundaries of $K$ with coefficient in $H$ and is denoted by  $ B^2_{ML}(K,H). $ 
The quotient group  $ \frac{Z^2_{ML}(K,H)}{ B^2_{ML}(K,H)}$ is called the second Lie cohomology of $K$ with coefficient in $H,$ and is denoted by  $ H^2_{ML}(K,H). $

Now, we have a natural projection homomorphism $\tilde{p} :  H^2_{ML}(K,H) \longrightarrow  H^2(K,H)$ defined by $$\tilde{p}((f,h)+B^2_{ML}(K,H)) = f +B^2(K,H)$$

Now onwards, we denote $(f,h)+B^2_{ML}(K,H)$ and $f+B^2(K,H)$ by $\bar{(f,h)}$ and $\bar{f}$, respectively.
\end{definition} 

Now we have the following results:

\begin{proposition}[Proposition 3.3, \cite{RLS}]\label{trivial}
	Let $H$ and $K$ be abelian groups with trivial multiplicative Lie algebra structures. Then $$ H^2_{ML}(K,H)\cong H^2(K,H)\times Hom(\wedge^2K,H).$$
\end{proposition}

\begin{proposition}[Proposition 3.4, \cite{RLS}]\label{kernel}
	The kernel of the homomorphism  $$ \tilde{p}:H^2_{ML}(K,H)\longrightarrow H^2(K,H) $$ is isomorphic to $ Hom(\frac{\wedge^2K}{J},H), $ where $\wedge^2K$ is the non abelian exterior square of $K$, and $J$ is the subgroup of $\wedge^2K$ generated by $$ \{((x\star y)\wedge^yz)((y\star z)\wedge^zx)((z\star x)\wedge^xy)\}. $$
\end{proposition}

\begin{remark}\label{sequence}
We have an exact sequence
$$1 \longrightarrow  Hom\bigg(\frac{\wedge^2K}{J},H\bigg) \longrightarrow H^2_{ML}(K,H)\overset{\tilde{p}}\longrightarrow H^2(K,H).$$
\end{remark}
 
\begin{corollary}[Corollary 3.5, \cite{RLS}]\label{improper}
Let $K$ be a group with  trivial Schur multiplier and improper multiplicative Lie algebra structure. Then $$H^2_{ML}(K,H)\cong H^2(K,H)\times Hom(\wedge^2K,H).$$
\end{corollary}

\begin{definition}
	Let $K$ be a finite multiplicative Lie algebra. The group $ H^2_{ML}(K,\C^*)$ is called the  Schur multiplier of $K$, and is denoted by $ \tilde M(K).$
\end{definition}

\begin{definition}\label{Lie exterior square}
	Let $ (K,\cdot,\star) $ be a multiplicative Lie algebra. Consider the multiplicative Lie algebra $ K \wedge^L K $ generated by the set $$ \{a\wedge b|a,b\in K \}\cup \{[a,b]_0|a,b\in K\} $$ of formal symbols subject to the folllowing relations :
	
	\begin{enumerate}
		\item $ 1 \wedge a = a \wedge 1 = a \wedge a = [1,a]_0 = [a,1]_0 = [a,a]_0 = 1 $
		
		\item $ (a \wedge b)(b \wedge a) = 1 = [a,b]_0[b,a]_0 $
		
		\item $ (a \wedge bc)^{-1}(a \wedge b)(^ba \wedge ^bc) = 1 = [a,bc]_0^{-1}[a,b]_0[^ba,^bc]_0 $
		
		\item $ (ab \wedge c)^{-1}(^ab \wedge ^ac)(a \wedge c) = 1 = [ab,c]_0^{-1}[^ab,^ac]_0[a,c]_0 $
		
		\item $ [a \wedge b,u \wedge v] = [[a,b]_0,(u \wedge v)] = [a,b]_0 \overset{\sim}{\star}(u \wedge v) = [a,b]_0\overset{\sim}{\star}[u,v]_0  $
		
		\item $ (a \wedge b)(^cb \wedge ^ca) = (^{ab}a^{-1} \wedge ^ac)(a \wedge c) $
		
		\item $ [a,b]_0[^cb,^ca]_0 = [^{ab}a^{-1},^ac]_0[a,c]_0 $
		
		\item $ [[a,b]_0,[u,v]_0] = [[a,b],[u,v]]_0 $
		
		\item $ (a \wedge b)\overset{\sim}{\star}(u \wedge v) = (a \star b)\wedge(u \star v) $
		
		where $ a,b,c,u,v \in K $.
	\end{enumerate}
\end{definition}

The multiplicative Lie algebra $K\wedge^L K$ will be termed as the Lie exterior square of $K.$ The Proposition \ref{MLA} ensures the existence of a unique homomorphism $\chi$ from $K\wedge^L K$ to $K$ given by $ \chi(a,b) = a\star b $ and $ \chi([a,b]_0) = [a,b].$ We have the following short exact sequence of multiplicative Lie algebras for a finite multiplicative Lie algebra $K:$ 
\begin{equation}\label{3}
1 \rightarrow \tilde M(K) \rightarrow K \wedge^L K \rightarrow (K \star K)[K,K] \rightarrow 1.
\end{equation}

\section{Schur Multiplier and Lie exterior square of a multiplicative Lie algebra}
The aim of this section is to find Schur multipliers and Lie exterior squares of some given multiplicative Lie algebras.

\begin{theorem} \label{simple}
	Let $K$ be a group.
\begin{enumerate}
\item Suppose $K$ is an abelian group with trivial multiplicative Lie algebra structure. Then the Lie exterior square $K\wedge^L K \cong  \tilde{M}(K)$ with trivial multiplicative Lie algebra structure. Moreover, $$ \tilde{M}(K)\cong M(K)\times M(K).$$

\item Suppose $K$ is a Lie simple group with  trivial Schur multiplier. Then $$ \tilde{M}(K)\cong Hom([K, K],\C^*).$$

\end{enumerate}
\end{theorem}
\begin{proof}
By sequence (\ref{3}) and Definition \ref{Lie exterior square}, we have $K\wedge^L K \cong  \tilde{M}(K)$ with trivial multiplicative Lie algebra structure. 
The rest follows from the Proposition \ref{trivial} and Corollary \ref{improper}.
\end{proof}

\begin{corollary}\label{Simple com}
Let $K$ be a group with trivial Schur multiplier and a multiplicative Lie algebra structure $\star$ such that $[K, K]$ is a  non abelian  simple group. Then $\tilde{M}(K)$ is trivial and $K\wedge^L K \cong [K, K]$ with the improper multiplicative Lie algebra structure.
\end{corollary}

\begin{proof}
	Since the Schur multiplier of $K$ is trivial, $K\wedge K \cong [K,K].$ By Remark \ref{sequence} $\tilde{M}(K) \cong Hom(\frac{[K,K]}{J},\C^*),$ where $J$ is the subgroup generated by $ \{((x\star y)\wedge^yz)((y\star z)\wedge^zx)((z\star x)\wedge^xy)\mid x, y, z \in K\}. $ Also, by Proposition \ref{equivarient}, the multiplicative Lie algebra structure $\star$ is uniquely determined by a $K$-equivariant homomorphism $\phi:[K,K] \rightarrow K.$ Since $[K,K]$ is simple, either $\phi$ is trivial or injective.
	
 If $\phi$ is trivial, then $J = \{1\}.$
 
Now, suppose $\phi$ is injective. Since $\phi([x\star y,^yz]) = (x\star y)\star ^yz,$, we have  $\phi([x\star y,^yz][y\star z,^zx][z\star x,^xy])  = 1 .$ So $J = 1.$ Hence, $\tilde{M}(K)$ is trivial. 

By sequence (\ref{3}) and Definition \ref{Lie exterior square}, we have $K\wedge^L K \cong  [K, K]$ with the improper multiplicative Lie algebra structure.
\end{proof}

\begin{corollary}
Let $K$ be a non abelian simple group with trivial Schur multiplier  and a multiplicative Lie algebra structure $\star$. Then $ \tilde{M}(K)$ is trivial and $K\wedge^L K \cong K$ with the improper multiplicative Lie algebra structure.
\end{corollary}

\begin{proof}
 Follows from Corollary \ref{Simple com}. 
\end{proof}

Now, we calculate the Schur multiplier of the Klein four group, the direct product of two cyclic group, the Dihedral group and the Quaternion group with arbitrary multiplicative Lie algebra structures from \cite{MS}.

\begin{proposition}\label{Klein}
Consider the Klein four group $V_4 = \langle a, b \mid a^2 = 1 = b^2, ab = ba\rangle$ with multiplicative Lie algebra structure $\star$ given by  $a \star b = a.$ Then $$ \tilde{M}(V_4)\cong V_4.$$
\end{proposition}
\begin{proof}
By Remark \ref{sequence}, we have $$1 \longrightarrow  Hom\bigg(\frac{M(V_4)}{J},\C^*\bigg) \longrightarrow \tilde{M}(V_4)\overset{\tilde{p}}\longrightarrow M(V_4),$$ where $J$ is the subgroup of $\wedge^2V_4$ generated by $$ \{((x\star y)\wedge^yz)((y\star z)\wedge^zx)((z\star x)\wedge^xy)\}. $$

\textbf{Claim:} $J = \{1\}$

For $ x = a , y = a , z = ab $, we have

$ ((a \star a) \wedge ab)((a \star ab) \wedge a)((ab \star a) \wedge a) = (1 \wedge ab)(a \wedge a)(a \wedge a) = 1.$ 

For $ x = a , y = b , z = ab $, we have

$ ((a \star b) \wedge ab)((b \star ab) \wedge a)((ab \star a) \wedge b) = (a \wedge ab)(a \wedge a)(a \wedge b) = (a \wedge b)(a \wedge b) = 1.$

Similarly, for any combination of $x, y$ and $z$, we can see that $$ ((x\star y)\wedge^yz)((y\star z)\wedge^zx)((z\star x)\wedge^xy) = 1.$$ Hence  $J = \{1\}.$

\textbf{Claim:} $\tilde{p}$ is surjective.

By identities (2) and (3) of Definition \ref{cocycle}, it is easy to see that $$ h(a,ab) = h(a,b) = h(ab,b) $$
$$h(b,ab)= h(ab,a) = h(b,a).$$

By using identity (4), we have $$ h(a,b)h(a,b)f(a,a) = 1 \hspace{4cm}  (i) $$

Now, since $ h(ab,ab) = 1 $ , we have $$ h(b,a)h(a,b)f(a,a) = 1 \hspace{4cm} (ii) $$ 

On comparing (i) and (ii), we have $h(a, b) = h(b, a)$. Therefore, $$ h(a,ab)  = h(ab,b) = h(b,ab)= h(ab,a) = h(b,a)  = h(a,b) = \sqrt{f(a,a)^{-1}}$$

Thus for any $f \in M(V_4)$, we have a map $h : V_4\times V_4 \to \mathbb{C}^*$ defined by $h(a,b) = \sqrt{f(a,a)^{-1}}$. So, the homomorphism $\tilde{p}$ is surjective. Hence, we have a short exact sequence 
$$1 \longrightarrow  Hom\left(\wedge^2V_4,\C^*\right) \longrightarrow \tilde{M}(V_4)\overset{\tilde{p}}\longrightarrow M(V_4)\longrightarrow 1$$

Now, the map $t: M(V_4) \to \tilde{M}(V_4)$ defined by $t(\bar{f}) =\bar{(f, h)}$ gives a splitting. Hence $\tilde{M}(V_4) \cong M(V_4) \times Hom(M(V_4),\C^*) \cong \Z_2 \times \Z_2\cong V_4.$
\end{proof}

\begin{remark}
Consider the Klein four group $V_4 = \langle a, b \mid a^2 = 1 = b^2, ab = ba\rangle$ with trivial multiplicative Lie algebra structure.  Then from the short exact sequence $(\ref{3})$, we have $$  V_4 \wedge^L V_4 \cong \tilde{M}(V_4) \cong V_4 $$ with trivial multiplicative Lie algebra structure.
\end{remark}

\begin{proposition}\label{Lie Klein}
	Consider the Klein four group $V_4 = \langle a, b \mid a^2 = 1 = b^2, ab = ba\rangle$ with multiplicative Lie algebra structure $\star$ given by  $a \star b = a.$ Then $$  V_4 \wedge^L V_4 \cong \Z_2 \times \Z_2 \times \Z_2 $$ with trivial multiplicative Lie algebra structure.
\end{proposition}

\begin{proof}
Since  $ \tilde M(V_4) \cong V_4 $ and $ (V_4 \star V_4)[V_4,V_4] \cong \Z_2 ,$ from the sequence $(\ref{3})$, we have
$$ 1 \rightarrow V_4 \rightarrow V_4 \wedge^L V_4 \rightarrow \Z_2 \rightarrow 1. $$
Since  $ \tilde M(V_4) \cong V_4 $ and $ (V_4 \star V_4)[V_4,V_4] \cong \Z_2 ,$ we have $ | V_4 \wedge^L V_4 | = 8 .$
From $(5)$ and $(8)$ of Definition \ref{Lie exterior square}, we have $ (a\wedge b)(u\wedge v) = (u\wedge v)(a\wedge b), [a,b]_0(u\wedge v) = (u\wedge v) [a,b]_0, $ and  $ [a,b]_0 [u,v]_0 =  [u,v]_0 [a,b]_0.$ So, $ V_4 \wedge^L V_4 $ is an abelian group of order $8$ which contains $V_4.$
This implies $ V_4 \wedge^L V_4 \cong \Z_2 \times \Z_2 \times \Z_2 $ or $ V_4 \wedge^L V_4 \cong \Z_4 \times \Z_2 .$ But from the Definition \ref{Lie exterior square},   order of every non trivial  element in $ V_4 \wedge^L V_4 $ is $2.$ Hence, $  V_4 \wedge^L V_4 \cong \Z_2 \times \Z_2 \times \Z_2 .$

Moreover, from $(5)$ and $(9)$ of Definition \ref{Lie exterior square}, we have $ [a,b]_0 \tilde{\star} [u,v]_0 = [a,b]_0\tilde{\star} (u\wedge v) = (a\wedge b)\tilde{\star} (u\wedge v) = 1. $ Hence $  V_4 \wedge^L V_4  $ is trivial multiplicative Lie algebra.
\end{proof}

\begin{proposition}\label{abeian}
	Consider the abelian group $K= \Z_m \times \Z_n $ of order $mn$ with a multiplicative Lie algebra structure $\star$. Then $$  \tilde{M}(K) = \Z_d \times \Z_d $$ where $ d = gcd(m,n). $
\end{proposition}

\begin{proof}
Suppose $K = \langle a, b \mid a^m = b^n =  1, ab = ba\rangle$ and  $a\star b = a^{i}b^{j}$, for some suitable $1\leq i \leq m, 1\leq j \leq n$ (for details see [Theorem 2.5 (2),  \cite{MS}]).
	
	By Remark \ref{sequence}, we have $$1 \longrightarrow  Hom\bigg(\frac{M(\Z_m \times \Z_n)}{J},\C^*\bigg) \longrightarrow \tilde{M}(\Z_m \times \Z_n)\overset{\tilde{p}}\longrightarrow M(\Z_m \times \Z_n),$$ where $J$ is the subgroup of $\wedge^2(\Z_m \times \Z_n)$ generated by $$ \{((x\star y)\wedge^yz)((y\star z)\wedge^zx)((z\star x)\wedge^xy)\}. $$
	
	\textbf{Claim:} $J = \{1\}.$
	
	For $ x = a , y = b , z = ab $, we have
	
	$ ((a \star b) \wedge ab)((b \star ab) \wedge a)((ab \star a) \wedge b) = (a^ib^j \wedge ab)(a^{-i}b^{-j} \wedge a)(a^{-i}b^{-j} \wedge b) = (b^j \wedge a)\linebreak (a^i \wedge b)(b^{-j}\wedge a)(a^{-i}\wedge b) = 1. $
	
	For $ x = a , y = a , z = ab $, we have
	
	$ ((a \star a) \wedge ab)((a \star ab) \wedge a)((ab \star a) \wedge a) = (1 \wedge ab)(a^ib^j \wedge a)(a^{-i}b^{-j} \wedge a) = 1.$ 
	
	Similarly, for all other combinations of $x, y$ and $z$, we can see that $$ ((x\star y)\wedge^yz)((y\star z)\wedge^zx)((z\star x)\wedge^xy) = 1.$$ 
	Hence  $J = \{1\}.$
	
	\textbf{Claim:} $\tilde{p}$ is surjective.
    
    By identities (1), (2) and (3) of Definition \ref{cocycle}, it is easy to see that 
	
	\begin{enumerate}
		\item $ h(x,1) = h(1,x) = h(x,x) = 1 $ for all $ x \in K. $ 
		This implies that $ h(a^i,a^j) = 1 = h(b^i,b^j) $
		
		\item $ h(a^p,a^qb^r) = h(a^p,b^r)f(a^{-q},a^q)^{-1} f(a^q,a^p\star b^r) f(a^q(a^p\star b^r),a^{-q}) $
		
		\item $ h(a^qb^r,a^p) = h(b^r,a^p)f(a^{-q},a^q)^{-1} f(a^q,b^r\star a^p) f(a^q(b^r\star a^p),a^{-q}) $
		
		\item $ h(b^p,a^qb^r) = h(b^p,a^q)f(a^{-q},a^q)^{-1} f(a^q,a^{-q}) $
		
		\item $ h(a^qb^r,b^p) = h(a^q,b^p) f(a^{-q},a^q)^{-1}f(a^q,a^{-q})$ 
		
		\item $ h(a^pb^q,a^rb^s)= h(a^pb^q,a^r)h(a^pb^q,b^s)f(a^{-r},a^r)^{-1}f(a^r,(a^pb^q)\star b^s)f(a^r((a^pb^q)\star b^s),a^{-r})\linebreak f((a^pb^q)\star a^r,(a^pb^q)\star b^s).   $
	\end{enumerate}

Also, we can see that $ h(a^p,b^r) $ can be written in terms of $ h(a,b)^{pr}. $

Now, since $ h(ab,ab) = 1. $ It can be seen  that $ h(b,a) $ can be written in terms of $ h(a,b). $

By identity (4) of the Definition \ref{cocycle}, we have
$$ h(a^ib^j,b)h(a^{m-i}b^{n-j},b)f(a^ib^j\star b,a^{-i}b^{-j}\star b) = 1. $$

This implies that $$  h(a^i,b)h(a^{m-i},b) f(a^{-i},a^i)^{-1} f(a^i,a^{-i})  f(a^{i},a^{m-1})^{-1}f(a^{m-i},a^{i})f(a^ib^j\star b,a^{-i}b^{-j}\star b) = 1. $$

So, $ h(a,b) $ can be written completely in terms of values of $f.$ Thus for any $f \in M(\Z_m \times \Z_n)$, we have a map $h : (\Z_m \times \Z_n)\times (\Z_m \times \Z_n) \to \mathbb{C}^*$ defined in terms of $f$. 
So, the homomorphism $\tilde{p}$ is surjective. Hence, we have a short exact sequence 
$$1 \longrightarrow  Hom(\wedge^2(\Z_m \times \Z_n),\C^*) \longrightarrow \tilde{M}(\Z_m \times \Z_n)\overset{\tilde{p}}\longrightarrow M(\Z_m \times \Z_n)\longrightarrow 1$$

Now, the map $t: M(\Z_m \times \Z_n) \to \tilde{M}(\Z_m \times \Z_n)$ defined by $t(\bar{f}) =\bar{(f, h)}$ gives a splitting. Hence $\tilde{M}(\Z_m \times \Z_n) \cong M(\Z_m \times \Z_n) \times Hom(M(\Z_m \times \Z_n),\C^*) \cong \Z_d \times \Z_d.$
\end{proof}

\begin{remark}
Proposition \ref{Klein} is the particular  case of Proposition \ref{abeian}. But here, we like to give the proof of Proposition \ref{Klein} so that one can easily understand the proof of the general case.
\end{remark}

\begin{proposition} Consider the Dihedral group $D_n = \langle a, b \mid a^2 = b^n =  1, aba = b^{-1}\rangle$ with a multiplicative Lie algebra structure $\star$.
\begin{enumerate}
\item If $n$ is odd, then $$ \tilde{M}(D_n)\cong \Z_n.$$
\item If $n$ is even, then $$ \tilde{M}(D_n)\cong \Z_2\times \Z_n.$$
\end{enumerate}
\end{proposition}
\begin{proof}
By Theorem 2.5 (3) of \cite{MS}, it is enough to assume that the multiplicative Lie algebra structure is given by $a\star b = b^{i}$, for any fix $1\leq i \leq n$.
	
	By Remark \ref{sequence}, we have $$1 \longrightarrow  Hom\bigg(\frac{\wedge^2(D_n)}{J},\C^*\bigg) \longrightarrow \tilde{M}(D_n)\overset{\tilde{p}}\longrightarrow M(D_n),$$ where $J$ is the subgroup of $\wedge^2D_n$ generated by $$ \{((x\star y)\wedge^yz)((y\star z)\wedge^zx)((z\star x)\wedge^xy)\}. $$
	
	\textbf{Claim:} $J = \{1\}$
	
	For $ x = a, y = b^p, z = ab^q $, we have
	
	$ ((b^p\star a) \wedge ^aab^q )((a\star ab^q) \wedge ^{ab^q}b^p )((ab^q\star b^p) \wedge ^{b^p}a ) = (b^{-pi}\wedge ab^{-q})(b^{pi}\wedge ab^{-2p}) = 1. $
	
	For $ x = y = a, z = b^p $, we have
	
	$ ((a\star a) \wedge ^ab^p )((a\star b^p) \wedge ^{b^p}a )((b^p\star a) \wedge ^aa ) = (b^{pi}\wedge ab^{-2p})(b^{-pi}\wedge a) = 1. $
	
	Similarly, for any combination of $x, y$ and $z$, we can see that $$ ((x\star y)\wedge^yz)((y\star z)\wedge^zx)((z\star x)\wedge^xy) = 1.$$ Hence  $J = \{1\}.$

\noindent\textbf{Case 1: n is odd}

In this case, ${M}(D_n)=\{1\}$. So $\tilde{M}(D_n) \cong Hom (\wedge^2(D_n),\C^*)\cong \mathbb{Z}_n$.

\noindent\textbf{Case 2: n is even} 

First of all, we show that $\tilde{p}$ is surjective. Using $(1),(2) $ and $(3)$ of the Definition \ref{cocycle}, we have the following:
	\begin{enumerate}
		
		\item $ h(x,1) = h(1,x) = h(x,x) = 1 $ for all $ x $ in $ D_n $.
		
		\item $ h(a,ab^j) = h(a,a)h(a,b^j)f(a,a)^{-1}f(a,b^{ji})f(ab^{ji},a)f(1,^ab^{ji}) = h(a,b^j)f(a,a)^{-1}f(a,b^{ji})\linebreak f(ab^{ji},a) $
		
		\item $ h(ab^j,a) = h(b^j,a)h(a,a)f(a,a)^{-1}f(a,b^{-ji})f(ab^{-ji},a)f(^ab^{-ji},1) = h(b^j,a)f(a,a)^{-1} \linebreak f(a,b^{-ji}) f(ab^{-ji},a) $
		
		\item $ h(b^j,ab^k) = h(b^j,a) $
		
		\item $ h(ab^k,b^j) = h(a,b^j) $
		
		\item $ h(ab^j,ab^k) = h(a,b^k)h(b^j,a)f(a,a)^{-2}f(a,b^{-ji})f(ab^{-ji},a)f(a,b^{ki})f(ab^{ki},a)f(b^{ji},b^{-ki}) $
		
		\item Since $ h(ab^j,ab^j) = 1 $,  
$$ h(a,b^j)h(b^j,a)f(a,a)^{-2}f(a,b^{-ji})f(ab^{-ji},a)f(a,b^{ji})f(ab^{ji},a)f(b^{ji},b^{-ji}) = 1 $$

\item $ h(a,b^2) = h(a,b)^2f(b^{-1},b)^{-1}f(b,b^i)f(bb^i,b^{-1})f(b^i,b^i).$ 
Similarly, $h(a,b^j)$ can be written in terms of $h(a,b)^j.$
	\end{enumerate}	
From above identities, it is clear that for any given $f$, image of any elements of $K \times K$  under the map $h$  is determined by the image of $(a, b)$ under the map $h$. 

Now, using $(5)$ of Definition \ref{cocycle}, we have \\
$ h(a,b^{-1}) = h(^aa,^ab) = h(a,b)f(a,a\star b)f(a,a)^{-1}f(a(a\star b),a) = h(a,b)f(a,b^i)f(a,a)^{-1}f(ab^i,a)$

$$h(a,b^{-1})= h(a,b)z, \hspace{4cm} (i) $$ where $ z = f(a,b^i)f(a,a)^{-1}f(ab^i,a) $

Since $ 1 = h(a,1) = h(a,bb^{-1})$, $h(a,b)h(a,b^{-1})f(b^{-1},b)^{-1}f(b,b^{-i})f(b^{-i+1},b^{-1})f(b^i,b^{-i})= 1$. By using (i), we have  $$h(a,b)h(a,b^{-1})w = h(a,b)h(a,b)zw, $$  where $ w = f(b^{-1},b)^{-1}f(b,b^{-i})f(b^{-i+1},b^{-1})f(b^i,b^{-i}).$

So, $$ h(a,b) = \sqrt{(zw)^{-1}}. $$

Thus for any $f \in M(V_4)$, we have a map $h : V_4\times V_4 \to \mathbb{C}^*$ defined by $h(a,b) = \sqrt{(zw)^{-1}}$. So, the homomorphism $\tilde{p}$ is surjective. Hence, we have a short exact sequence 
$$1 \longrightarrow  Hom(\wedge^2D_n,\mathbb{C}^*) \longrightarrow \tilde{M}(D_n)\overset{\tilde{p}}\longrightarrow M(D_n)\longrightarrow 1$$

Now, the map $t: M(D_n) \to \tilde{M}(D_n)$ defined by $t(\bar{f}) =\bar{(f, h)}$ gives a splitting. 

Hence $\tilde{M}(D_n) \cong M(D_n) \times Hom(\wedge^2D_n,\mathbb{C^*}) \cong \mathbb{Z}_2 \times \mathbb{Z}_n.$
\end{proof}

\begin{proposition}
	 Consider the Dihedral group $D_n = \langle a, b \mid a^2 = b^n =  1, aba = b^{-1}\rangle$ (where $n$ is odd) with a multiplicative Lie algebra structure $a\star b = b^i$ for any fix $1\leq i \leq n$. Then $  D_n \wedge^L D_n \cong \Z_n \times \Z_n $ is trivial multiplicative Lie algebra.
\end{proposition}

\begin{proof}
    Since $M(D_n) \cong \{1\}$ and $\tilde{M}(D_n) \cong \Z_n$, $D_n \wedge D_n \cong \Z_n$ and by the sequence (\ref{3}), we have $$1 \rightarrow \Z_n \rightarrow D_n \wedge^L D_n \rightarrow D_n \wedge D_n \rightarrow 1.$$ So, $|D_n \wedge^L D_n| = n^2.$ From $(5)$ and $(8)$ of Definition \ref{Lie exterior square}, we have $ (a\wedge b)(u\wedge v) = (u\wedge v)(a\wedge b), [a,b]_0(u\wedge v) = (u\wedge v) [a,b]_0, $ and  $ [a,b]_0 [u,v]_0 =  [u,v]_0 [a,b]_0.$ So, $ D_n \wedge^L D_n $ is an abelian group of order $n^2,$ which contains $\Z_n.$
	This implies $ D_n \wedge^L D_n \cong \Z_n \times \Z_n$.
	
	Moreover, from $(5)$ and $(9)$ of Definition \ref{Lie exterior square}, we have $ [a,b]_0 \tilde{\star} [u,v]_0 = [a,b]_0\tilde{\star} (u\wedge v) = (a\wedge b)\tilde{\star} (u\wedge v) = 1. $ Hence, $  D_n \wedge^L D_n  $ is a trivial multiplicative Lie algebra.
	
\end{proof}

\begin{proposition}
	Consider the quaternion group $ Q_n = \langle a,b : a^2 = b^n, aba^{-1} = b^{-1}\rangle $ of order $4n$ with multiplicative Lie algebra structure $ a\star b = b^i,$ where $1\le i\le 2n-1.$ Then $\tilde{M}(Q_n) \cong \mathbb{Z}_{{\gcd(n,i)}}$.
\end{proposition}

\begin{proof}
	We know that $ M(Q_n) = {1}.$ From the short exact sequence 
	$$ 1 \longrightarrow M(Q_n) \longrightarrow  Q_n \wedge Q_n \longrightarrow [Q_n,Q_n] \longrightarrow 1,$$
    we have $ Q_n \wedge Q_n \cong [Q_n,Q_n] \cong \langle b^2 \rangle \cong \mathbb{Z}_n.$
	By Remark \ref{sequence}, we have $$1 \longrightarrow  Hom\bigg(\frac{\wedge^2(Q_n)}{J},\mathbb{C}^*\bigg) \longrightarrow \tilde{M}(Q_n)\overset{\tilde{p}}\longrightarrow M(Q_n),$$ where $J$ is the subgroup of $\wedge^2Q_n$ generated by $$ \{((x\star y)\wedge^yz)((y\star z)\wedge^zx)((z\star x)\wedge^xy)\}. $$
	
	Since, $ M(Q_n) = 1, $ we have $ \tilde{M}(Q_n) \cong Hom\bigg(\frac{\wedge^2(Q_n)}{J},\mathbb{C}^*\bigg)  \cong Hom\bigg(\frac{\mathbb{Z}_n}{J},\mathbb{C}^*\bigg). $
		
	\textbf{Claim:} $J = \mathbb{Z}_{m}, $ where $ m = \frac{n}{\gcd(n,i)} .$
	
	\noindent\textbf{Case 1:}
	For  $ x =a, y = b^p , z = ab^q$, we have
	
	$ [a\star b^p,^{b^p}ab^q] [b^p\star ab^q,^{ab^q} a] [ab^q\star a,^a b^p]  
	= [b^{ip},^{b^p}ab^q] [b^{-ip},^{ab^q} a] [^ab^{-iq},^ab^p]  $
	
	$ = b^{2ip}\cdot 1 \cdot 1 = b^{2ip}  $
	
	\noindent\textbf{Case 2:}
	For  $ x = y = a, z = b^p $, we have
	
	$ [a\star a,^ab^p] [a\star b^p,^{b^p} a] [b^p\star a,^a a]  
	= [1,^ab^p] [b^{ip},^{b^p} a] [b^{-ip},^a a]  $
	
	$ = 1 \cdot  b^{2i} \cdot  b^{-2i} = 1  $
	
	\noindent\textbf{Case 3:}
	For  $ x = y = a, z = ab^p $, we have
	
	$ [a\star a,^aab^p] [a\star ab^p,^{ab^p} a] [ab^p\star a,^a a] 
	=  [1,^aab^p] [^ab^{ip},^{ab^p} a] [^ab^{-ip},^a a] $
	
	$ = 1 \cdot  b^{-2ip} \cdot  b^{2ip} = 1   $
	
	Similarly, for all other combinations of $x, y$ and $z$, we can see that 
	
	$ ((x\star y)\wedge^yz)((y\star z)\wedge^zx)((z\star x)\wedge^xy) = 1 $ or $b^{2ip} .$
	
	So, $ J = \langle (b^2)^{i} \rangle \cong \mathbb{Z}_\frac{n}{\gcd(n,i)}$. Hence, $\frac{\wedge^2(Q_n)}{J}\cong \mathbb{Z}_{\gcd(n,i)}$. Therefore, $ \tilde{M}(Q_n) \cong  \mathbb{Z}_{\gcd(n,i)}.  $	
\end{proof}

\newpage
Table \ref{table} provides the Schur multiplier of some precise multiplicative Lie algebras.
\begin{table}[H]
	\begin{tabular}{|p{5cm}|p{2cm}|p{2.5cm}|p{2.5cm}|p{2cm}|}
		\hline 
	\smallskip \hspace{1cm}	Group $K$ & \smallskip $\wedge^2K $ & \smallskip $M(K)$ & Multiplicative Lie algebra structure & \smallskip $ \tilde M(K)$ \\
		\cline{1-5}
		
		  $\mathbb{Z}_n$ & trivial & trivial & Lie simple & trivial  \\
		\cline{1-5}

		
		$ V_4 $ & $\mathbb{Z}_2$ & $\mathbb{Z}_2$  & arbitrary & $V_4$ \\	
		\cline{1-5}
%
			$ SL(2,3)$ & $Q_2$ &  trivial & Lie simple & $V_4$ \\
		\cline{1-5}
		
	$Q_n = \langle a, b ~\mid~ a^2 = b^n , aba^{-1} = b^{-1}\rangle$  &  $\mathbb{Z}_n$ & trivial & $a\star b = b^i$&  $\mathbb{Z}_{\gcd(n, i)}$ \\
			\cline{1-5}
			$D_n = \langle a, b ~\mid~ a^2 = b^n = 1, aba = b^{-1}\rangle$  &  $\mathbb{Z}_n$ & trivial if $n$ is odd and  $\mathbb{Z}_2$ if $n$ is even &  arbitrary&  $\mathbb{Z}_n$ if  $n$ is odd and $ \mathbb{Z}_n \times \mathbb{Z}_2$ if $n$ is even. \\
			\cline{1-5}
			
			$ \mathbb{Z}_m \times \mathbb{Z}_n = \langle a, b ~\mid~ a^m = b^n = 1, ab = ba\rangle$ & $\mathbb{Z}_d,$ where $d=(m,n)$ & $\mathbb{Z}_d,$ where $d=(m,n)$ &  arbitrary   & $\mathbb{Z}_d \times \mathbb{Z}_d$ \\
			\cline{1-5}
%
		$Q_2 \times \mathbb{Z}_p$ & $\mathbb{Z}_2$ & trivial &  Lie simple & $\mathbb{Z}_2$ \\
		\cline{1-5}
		\smallskip
		$\langle a, b ~\mid~ a^8 = b^p = 1, a^{-1}ba = b^{-1}\rangle$  & \smallskip $\mathbb{Z}_p$ & \smallskip trivial & \smallskip Lie simple & \smallskip $\mathbb{Z}_p$ \\
		
		\cline{1-5}
		\smallskip
		$\langle a, b ~\mid~ a^8 = b^p = 1, a^{-1}ba = b^{\alpha}\rangle$, where $\alpha$ is a primitive root of $\alpha^4 \equiv 1$ (mod $p$), $4 | (p-1)$  & \smallskip $\mathbb{Z}_p$ & \smallskip trivial & \smallskip  Lie simple& \smallskip $\mathbb{Z}_p$ \\
		
		\cline{1-5}
		\smallskip
		$\langle a, b ~\mid~ a^8 = b^p = 1, a^{-1}ba = b^{\alpha}\rangle$, where $\alpha$ is a primitive root of $\alpha^8 \equiv 1$ (mod $p$), $8 | (p-1)$  & \smallskip $\mathbb{Z}_p$ & \smallskip trivial & \smallskip Lie simple & \smallskip $\mathbb{Z}_p$ \\
		\cline{1-5}
	
		\smallskip
		$\langle a, b, c, d ~\mid~ a^4 = b^2=c^2 = d^p = 1, ab = ba, ac = ca, bc = cb, d^{-1}ad = b, d^{-1}bd = c, d^{-1}cd = ab\rangle$ & \smallskip  $\mathbb{Z}^3_2$  & \smallskip trivial & \smallskip Lie simple  & \smallskip $\mathbb{Z}^3_2$ \\

		\hline
	\end{tabular}	
\caption{\label{table}$\tilde M(K) $ for some  groups with arbitrary multiplicative Lie algebra structure.}
\end{table}

\vspace{.5cm}

\noindent{\bf Acknowledgement:}
The first named author sincerely thanks IIIT Allahabad and University grant commission (UGC), Govt. of India, New Delhi for research fellowship. The second named author sincerely thanks IIIT Allahabad and Ministry of Education, Government of India for providing institute fellowship. 


\end{document}